\documentclass[a4paper]{article}
\usepackage{amsthm,amsmath,amsfonts,graphicx}
\usepackage[T1]{fontenc}
\usepackage{amssymb}
\usepackage{todonotes}
\usepackage{esvect}
\usepackage{complexity}
\usepackage[export]{adjustbox}

\def\N{\mathbb{N}}

\newcommand{\revvec}[1]{\overset{\tiny\leftarrow}{#1}}
\renewcommand{\vec}[1]{\overset{\tiny\rightarrow}{#1}}

\usepackage[font=small,labelfont=bf]{caption}
\usepackage[font=small,labelfont=normalfont,labelformat=simple]{subcaption}

\newtheorem{theorem}{Theorem}

\newtheorem{observation}{Observation}
\newtheorem{problem}{Problem}
\newtheorem{definition}{Definition}
\newtheorem{proposition}{Proposition}
\newtheorem{corollary}{Corollary}
\newtheorem{lemma}{Lemma}

\title{Colorings of oriented planar graphs avoiding a monochromatic subgraph}
\date{}
\author  {
  Helena Bergold \footnotemark[1]\; \footnotemark[2]
  \and
  Winfried Hochst\"{a}ttler \footnotemark[2]
  \and
  Raphael Steiner \footnotemark[3]}

\begin{document}
\maketitle

  \renewcommand{\thefootnote}{\fnsymbol{footnote}}
  \footnotetext[1]{Institut f\"{u}r Informatik, Freie Universit\"at Berlin, Germany, email: \texttt{helena.bergold@fu-berlin.de}.
  	Funded by DFG-GRK 2434 Facets of Complexity.}
  \renewcommand{\thefootnote}{\arabic{footnote}}
  \footnotetext[2]{Fakult\"{a}t f\"{u}r Mathematik und Informatik, FernUniversit\"{a}t in Hagen, Germany, email: \texttt{winfried.hochstaettler@fernuni-hagen.de}.}
  \footnotetext[3]{Institut f\"{u}r Mathematik, Technische Universit\"at Berlin, Germany, email: \texttt{steiner@math.tu-berlin.de}.
  Funded by DFG-GRK 2434 Facets of Complexity.}
  \renewcommand{\thefootnote}{\arabic{footnote}}

\setlength{\parindent}{0em}
\begin{abstract}
	For a fixed simple digraph $F$ and a given simple digraph $D$, an \emph{$F$-free $k$-coloring} of $D$ is a vertex-coloring in which no induced copy of $F$ in $D$ is monochromatic. We study the complexity of deciding for fixed $F$ and $k$ whether a given simple digraph admits an $F$-free $k$-coloring. Our main focus is on the restriction of the problem to planar input digraphs, where it is only interesting to study the cases $k \in \{2,3\}$. From known results it follows that for every fixed digraph $F$ whose underlying graph is not a forest, every planar digraph $D$ admits an $F$-free $2$-coloring, and that for every fixed digraph $F$ with $\Delta(F) \ge 3$, every oriented planar graph $D$ admits an $F$-free $3$-coloring. 

	We show in contrast, that
	\begin{itemize}
		\item if $F$ is an orientation of a path of length at least $2$, then it is \NP-hard to decide whether an \emph{acyclic and planar} input digraph $D$ admits an $F$-free $2$-coloring.
		\item if $F$ is an orientation of a path of length at least $1$, then it is \NP-hard to decide whether an \emph{acyclic and planar} input digraph $D$ admits an $F$-free $3$-coloring.
	\end{itemize}
\end{abstract}

\section{Introduction}
In this paper we are concerned with colorings of simple planar digraphs, i.e., orientations of planar graphs.
We consider the following notion of coloring.
\begin{definition}
Let $F$ be a digraph. 
\begin{itemize}
\item A digraph is \emph{$F$-free} if it has no induced subdigraph isomorphic to $F$.
\item A coloring $c\mathop{:}V(D) \rightarrow \{0,1,\ldots,k-1\}$ of the vertices of a digraph $D$ is called an \emph{$F$-free $k$-coloring} or simply \emph{$F$-free} if for every color $i \in \{0,1,\ldots,k-1\}$, each of the subdigraphs $D[c^{-1}(i)]$, $i=0,1,\ldots,k-1$ induced by the color classes is $F$-free.
\end{itemize}
\end{definition}

There is an analogous definition of $F$-free coloring for undirected graphs, which has been the subject of study of several previous papers. Gimbel and Hartman~\cite{GIMBEL2003} studied \emph{subcolorings} of graphs, which are colorings in which every color class induces a disjoint union of cliques, or, equivalently, in which no induced $P_3$ is monochromatic. Generalizing this setting, Gimbel and Ne\v{s}et\v{r}il~\cite{Gimbel2010} studied the complexity of \emph{cograph colorings} of graphs, i.e., vertex-colorings in which every color class induces a cograph. As the latter are exactly the graphs avoiding $P_4$ as an induced subgraph, these are the $P_4$-free colorings. In both cases hardness results were obtained even for planar graphs: It was shown that deciding whether a given planar graph admits a $2$-subcoloring, a $3$-subcoloring, a $2$-cograph coloring, or a $3$-cograph-coloring, respectively, are each \NP-complete. All these results were further generalized by Broersma, Fomin, Kratochvil and Woeginger~\cite{Broersma2005} who gave a complete dichotomous description of the complexity of coloring planar graphs avoiding a monochromatic induced copy of some connected planar graph $F$. They proved that the $2$-coloring problem is \NP-hard if $F$ is a tree consisting of at least two edges and polynomially-solvable in all other cases. The $3$-coloring problem was shown to be NP-hard if $F$ is a path of positive length and polynomially solvable in all remaining cases.

In this paper, we will study the directed version of this problem. For a given digraph $F$, we study the complexity of
$F$-free $k$-coloring for planar digraphs as input.
If there is an $U(F)$-free coloring of a graph, we can also orient the edges arbitrarily and get an $F$-free coloring of an orientation.  Here $U(F)$ is the underlying undirected graph of the digraph $F$.  We determine the
complexity of $P$-free coloring of planar digraphs with $2$ or
$3$ colors for every oriented path $P$. More generally, for a
fixed digraph $F$ we would like to determine the complexity of the
following decision problem.

\begin{problem}[PLANAR $F$-FREE $k$-COLORING, $k$-$F$-PFC]
	Given a planar digraph $D$, determine whether or not $D$ admits an $F$-free coloring using $k$ colors.
\end{problem}

Note that the only interesting variants of the problem are for
$k \in \{2,3\}$, as clearly for $k \ge 4$ and any non-trivial digraph
$F$ on at least two vertices the $4$-color theorem implies that every
simple planar digraph admits an $F$-free coloring with $k$ colors. On the
other hand, if $|V(F)|=1$ or $k=1$, the problem admits (trivial)
polynomial algorithms. Finally, if $F$ consists of isolated vertices
only, no planar digraph on more than $4k\,|V(F)|$ vertices can have an 
$F$-free $k$-coloring. The reason for this is as follows: In an $F$-free $k$-coloring of a planar digraph $D$, every color class induces a planar subdigraph which contains no stable set of size $|V(F)|$. By the $4$-color-theorem, this means that this subdigraph has at most $4(|V(F)|-1)<4|V(F)|$ vertices. Hence, in total a planar digraph which admits an $F$-free $k$-coloring has order at most $k\, \cdot\, 4|V(F)|$. 
Hence, since relevant instances of this problem are bounded in size, any (brute-force) algorithm responds in constant time.

\begin{observation}
	The $k$-$F$-PFC is polynomial-time solvable for $k \ge 4$ and every digraph $F$.
\end{observation}

Using the following fact from the literature we can further restrict our attention to the case where $F$ is an oriented tree.

\begin{proposition}[\cite{thomassencycles}, \cite{thomassentriangles}]
	Let $G$ be a planar (undirected) graph, and let $n \ge 3$. Then $G$ admits a $C_n$-free $2$-coloring.
\end{proposition}
\begin{proof}
The case $n=3$ is well-known and can be found for instance in \cite{thomassentriangles}, while the cases $n \ge 4$ are covered by \cite{thomassencycles}.
\end{proof}

\begin{corollary}
The $k$-$F$-PFC is polynomial-time solvable for every $k \in \mathbb{N}$ and every digraph $F$ containing a (not necessarily directed) cycle.
\end{corollary}

For the case of $k=3$ colors, a more general statement holds true.
\begin{proposition}
	Let $D$ be a planar digraph and $F$ an orientation of a graph of maximum degree at least $3$. Then $D$ admits an $F$-free $3$-coloring.
\end{proposition}
\begin{proof}
The main result of~\cite{linear3} states that the vertex set of every planar graph can be partitioned into $3$ parts, each of them inducing a disjoint union of paths. This clearly yields an $F$-free coloring of every planar digraph $D$, hence we can solve the problem by returning 'true' for every planar input digraph $D$.
\end{proof}

On the negative side, one can show that the $F$-free $2$- and $3$-coloring problem becomes hard if $F$ is an orientation of a path. More precisely, the following are our main results. By $\mathcal{P}_n$ we denote the set of all orientations of the path $P_n$ on $n \in \N$ vertices. 

\begin{theorem} \label{thm:2col_NPhard}
Let $n \in \mathbb{N}$ and $\vec{P} \in \mathcal{P}_n$. Then the $2$-$\vec{P}$-PFC is polynomial-time solvable for $n \le 2$ and \NP-complete for $n \ge 3$, even restricted to acyclic digraphs.
\end{theorem}

\begin{theorem}\label{thm:3col_NPhard}
	Let $n \in \mathbb{N}$ and $\vec{P} \in \mathcal{P}_n$. Then the $3$-$\vec{P}$-PFC is polynomial-time solvable for $n=1$ and \NP-complete for $n \ge 2$, even restricted to acyclic digraphs.
\end{theorem}

In particular this means for directed paths $P$ of arbitrary length, there are orientations of planar graphs which  require four colors for a $P$-free coloring. This is in contrast to Neumann-Lara's Two-Color-Conjecture that any orientation of a planar graph admits an $2$-coloring without monochromatic directed cycle~\cite{NeumannLara1982}.

The proofs of Theorem~\ref{thm:2col_NPhard} and Theorem~\ref{thm:3col_NPhard} are separated into two steps, first we reduce the problems to cases of small $n$ in Section~\ref{sec:remove} and then describe explicit \NP-hardness reductions for paths of lengths $1$, $2$ and $3$ in Sections~\ref{sec:basecases2},~\ref{sec:basecases25}, and~\ref{sec:basecases3}.
In Section \ref{sec:outer} we show for every orientation $\vec{P}$ of a path the existence of an acyclic outerplanar digraph without a $\vec{P}$-free $2$-coloring. We use this side result later in the \NP-hardness reductions.

Note that the complexity of the $2$-$F$-PFC where the underlying undirected tree of $F$ has maximum degree at least $3$ remains partially open. We show that the problem is \NP-hard if the graph after successively deleting all leaves is a path of length $2$ or $3$. The complexity of those graphs where we get a star $K_{1,n}$ is not determined.

\paragraph{\textbf{Notation}} All digraphs considered in this paper are simple, that is, they do not contain loops nor parallel or anti-parallel arcs. An arc (or directed edge) in a digraph $D$ with tail $u$ and head $v$ is denoted by $(u,v)$. By $V(D)$ and $A(D)$ we denote, respectively, the set of vertices and arcs in the digraph $D$. If $D$ is a directed graph, the \emph{underlying graph} $U(D)$ is the undirected graph obtained from $D$ by ignoring the orientations of the edges. Vice-versa, if $G$ is an undirected graph, then any (simple) digraph $D$ with $U(D)=G$ is called an \emph{orientation} of $G$. A \emph{proper coloring} of an undirected graph $G$ is an assignment of colors to the vertices such that adjacent vertices receive distinct colors.

\section{Outerplanar digraphs forcing a monochromatic oriented path}\label{sec:outer}

In this section, we prove the following auxiliary result, which states the existence of acyclic outerplanar digraphs in which every $2$-coloring induces a monochromatic induced copy of a given oriented path.
We denote by an \emph{outerplanar digraph} an orientation of an outerplanar graph (i.e. a graph which can be drawn in the plane with no edges intersecting and all vertices on the outer face).

\begin{theorem}\label{outermonchromaticpaths}
	Let $\vec{P}$ be an oriented path. Then there exists an acyclic outerplanar digraph $\vec{O}(\vec{P})$ with the property that every $2$-coloring of its vertices contains a monochromatic induced copy of $\vec{P}$. 
\end{theorem}

\begin{figure}[htb]
	\centering
	\includegraphics[scale = 0.6,page = 1]{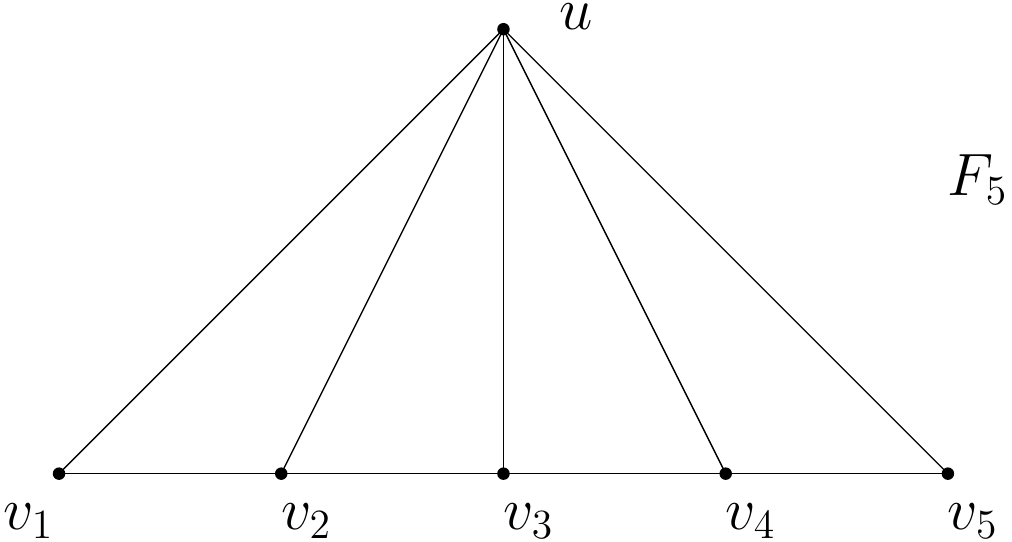}
	\caption{The fan $F_5$ with a path of length $4$ and the root $u$.}
	\label{fig:Fan}
\end{figure}

\begin{figure}
	\centering
	\includegraphics[scale = 0.5, page = 2]{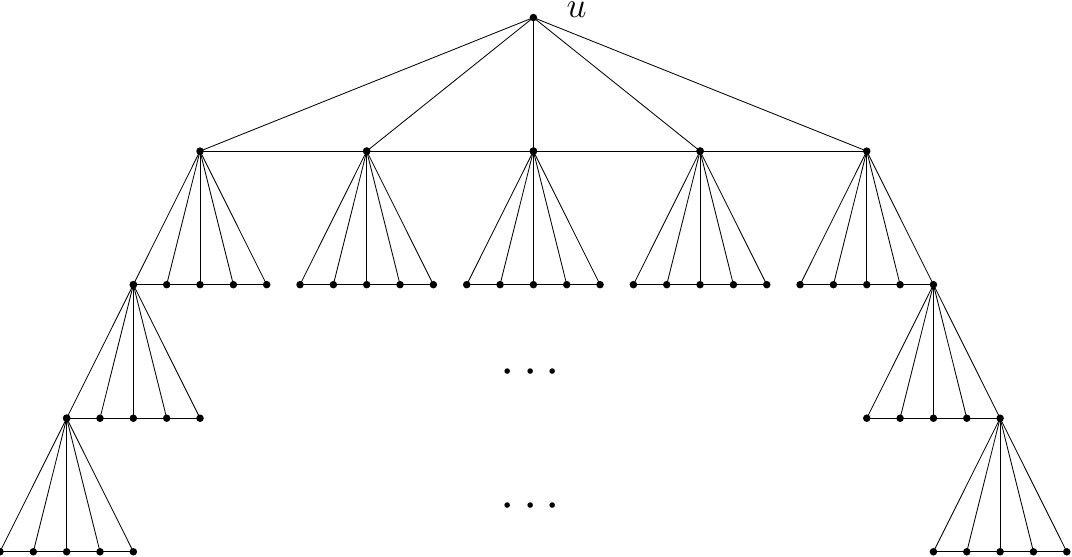}
	\caption{Illustration of $O^{4}$ with root $u$. }
	\label{fig:TreeOfFans}
\end{figure}
\begin{proof}
	Let $\vec{P} \in \mathcal{P}_{\ell+1}$ an oriented path of length $\ell$.
	We construct an outerplanar graph starting with a fan $F = F_{\ell +1}$ consisting of a path on $\ell +1$ vertices $v_1, \ldots, v_{\ell+1}$ and an additional vertex $u$ which is adjacent to every vertex of the path. We call this additional vertex $u$ the root. For an illustration see Figure~\ref{fig:Fan}.
	Starting from this fan $O^{1} = F$ with root $u$ we recursively 
	construct a graph $O^{k+1}$ from the graph $O^k$ for every $k \leq \ell -1$. 
	For this we add a copy $F_v$ of the fan for every vertex $v \in V(O^k)$ which has distance $k$ from the vertex $u$. 
	The resulting graph $O^{\ell}$ (see Figure \ref{fig:TreeOfFans}) has vertices with distance $\ell$ from $u$. 
	Clearly this graph $O^{\ell}$ is outerplanar. 
	
	We now orient the edges of the graph $O^{\ell}$ depending on the oriented path $\vec{P}$. 
	The edges of the path in every copy of the fan $F$ are oriented in such a way that they are isomorphic to the considered path $\vec{P}$. Furthermore in every copy of the fan we orient the arcs from the root vertex of this copy to the other vertices in such a way that the induced path connecting every vertex with distance $\ell$ to the vertex $u$ of $O^{\ell}$ is isomorphic to the considered path $\vec{P}$. 
	We call the so constructed acyclic outerplanar digraph $\vec{O}(\vec{P})$.
	
	We now consider 2-colorings of $\vec{O}(\vec{P})$ and show that every $2$-coloring contains an induced monochromatic subgraph isomorphic to $\vec{P}$. 
	Since there is an induced copy of $\vec{P}$ in every copy of the fan, we need two colors to color this induced path in order to avoid a monochromatic induced $\vec{P}$. 
	Now we look at the color $c(u) \in \{ 0,1\}$ of the root. In every induced path isomorphic to $\vec{P}$ in a copy of a fan $F_v$ at least one vertex is colored with $c(u)$. Hence there is a monochromatic path from $u$ to a vertex with distance $\ell$ from $u$. Since we defined the orientation in such a way that these path are isomorphic to $\vec{P}$, every 2-coloring of $\vec{O}(\vec{P})$ contains a monochromatic $\vec{P}$.
\end{proof}

\section{Removing leaves}\label{sec:remove}

In this section, we show the following two results, which will be used as the ``inductive steps'' in the proofs of Theorem~\ref{thm:2col_NPhard} and Theorem~\ref{thm:3col_NPhard}, respectively.
We need the following notation: Given an oriented tree $\vec{T}$, we denote by $\rm{lrem}(\vec{T})$ the oriented tree obtained from $\vec{T}$ by deleting all its leaves.

\begin{proposition}\label{remove2col}
	Let $\vec{T}$ be an oriented tree. Then the $2$-${\rm lrem}(\vec{T})$-PFC reduces polynomially to the $2$-$\vec{T}$-PFC. This remains true for the restrictions of both problems to acyclic digraphs.
\end{proposition}

We obtain a similar result for the $3$-coloring problem. 

\begin{proposition}\label{remove3col}
	Let $\vec{P} \in \mathcal{P}_n$, $n \ge 4$. Then the $3$-${\rm lrem}(\vec{P})$-PFC reduces polynomially to the $3$-$\vec{P}$-PFC. This remains true for the restrictions of both problems to acyclic digraphs.
\end{proposition}

\begin{proof}[Proof of Proposition~\ref{remove2col}]
	Suppose we are given a planar digraph $D$ as an instance of the $2$-${\rm lrem}(\vec{T})$-PFC. 
	Let $\ell$ be the number of leaves of $\vec{T}$. We construct a digraph $D'$ obtained from $D$ by adding for each vertex $v \in V(D)$ the disjoint copies $\vec{T}_{1,v}^+, \ldots, \vec{T}_{\ell,v}^+, \vec{T}_{1,v}^-, \ldots ,\vec{T}_{\ell,v}^- $ of $\vec{T}$ to $D$ and adding for all $2\ell$ disjoint copies all the arcs $(v,u)$, $u \in V(\vec{T}_{j,v}^+)$ and $(u,v)$, $u \in V(\vec{T}_{j,v}^-)$, $j = 1, \ldots, \ell$. The digraph $D'$ is still planar, since each of the attached copies $\vec{T}_{j,v}^+, \vec{T}_{j,v}^-$, $v \in V(D)$, $j = 1, \ldots, \ell$ is a tree and therefore an outerplanar graph. 

\paragraph{\textbf{Claim}} $D'$ admits a $\vec{T}$-free $2$-coloring if and only if $D$ admits a ${\rm lrem}(\vec{T})$-free $2$-coloring. 

	\begin{proof}
		Suppose for the first direction we are given a $\vec{T}$-free $\{0,1\}$-coloring $c'$ of $D'$. Let $c\mathop{:=}c'|_{V(D)}$ be the $2$-coloring induced by $c'$ on $D$. We claim that $c$ is ${\rm lrem}(\vec{T})$-free. Towards a contradiction assume that for some $i \in \{0,1\}$ there is $X \subseteq c^{-1}(i) \subseteq V(D)$ such that $D[X]$ is isomorphic to ${\rm lrem}(\vec{T})$. Now, since $c'$ is a $\vec{T}$-free coloring of $D'$, in each of the copies $\vec{T}_{j,v}^+, \vec{T}_{j,v}^-$ for $v \in V(D)$ and $j = 1, \ldots, \ell$ there must exist vertices of both colors. By picking a vertex of color $i$ from each of these copies and adding them to $X$ we obtain a monochromatic vertex-set $X'$ in $(D',c')$ such that $D'[X']$ is isomorphic to the digraph obtained from ${\rm lrem}(\vec{T})$ by attaching to every vertex $t \in V({\rm lrem}(\vec{T}))$ new vertices $t_1^+, \ldots, t_{\ell}^+, t_1^-, \ldots, t_{\ell}^-$ with the arcs $(t,t_j^+)$ and $(t_j^-,t)$ for all $j = 1, \ldots, \ell$. Clearly, this digraph contains a copy of $\vec{T}$ as an induced subdigraph and hence $(D',c')$ contains a monochromatic copy of $\vec{T}$, which contradicts our assumption on $c'$. Hence, $c$ is indeed ${\rm lrem}(\vec{T})$-free. This shows the first implication. 

		For the reverse direction, suppose $c\mathop{:}V(D) \rightarrow \{0,1\}$ is an ${\rm lrem}(\vec{T})$-free $2$-coloring of $D$. We extend this to a $2$-coloring $c'$ of $D'$ by properly coloring the vertices within each copy $\vec{T}_{j,v}^+,\vec{T}_{j,v}^-$, $j = 1\ldots, \ell$, of $\vec{T}$ according to the bipartition of the underlying tree $T$. 
		We claim that $c'$ is $\vec{T}$-free. Suppose that for some $i \in \{0,1\}$ there is $X' \subseteq V(D')$ such that $D'[X']$ is isomorphic to $\vec{T}$. By definition of $c'$, every vertex in $X' \setminus V(D)$ is incident to at most one monochromatic arc and hence must be a leaf of $D'[X']$. We conclude that $X\mathop{:=}X' \cap V(D)$ is a monochromatic vertex-set in $(D,c)$ such that $D[X]$ is isomorphic to a digraph obtained from $\vec{T}$ by removing some of its leaves. Hence, $D[X]$ contains a monochromatic copy of ${\rm lrem}(\vec{T})$ as an induced subgraph. However, this contradicts our initial assumption on the coloring $c$ of $D$, and shows that indeed $c'$ defines a $\vec{T}$-free $2$-coloring of $D'$. This finishes the proof of the claimed equivalence.  \phantom\qedhere
		\hfill $\triangle$
	\end{proof}

	Since the sizes of $D$ and $D'$ are linearly related, we have found a polynomial reduction of the $2$-${\rm lrem}(\vec{T})$-PFC to the $2$-$\vec{T}$-PFC. This concludes the proof of the first part of the proposition. For the second claim in the proposition it suffices to verify that $D'$ is acyclic if and only if $D$ is acyclic, since then we can use the same polynomial reduction to also reduce the $2$-${\rm lrem}(\vec{T})$-PFC restricted to acyclic inputs to the $2$-$\vec{T}$-PFC with acyclic inputs. However, this directly follows since $D$ is an induced subdigraph of $D'$, and since each of the copies $\vec{T}_{j,v}^+,\vec{T}_{j,v}^-$, $v \in V(D)$, $j = 1 , \ldots, \ell$ of $\vec{T}$ themselves are clearly acyclic and separated from the rest of the graph by directed edge-cuts. This shows the second claim in the proposition and concludes the proof.
\end{proof}

The following proof of Proposition~\ref{remove3col} works analogously to the previous proof, except that we attach copies of the outerplanar digraphs described in Section~\ref{sec:outer}.

\begin{proof}[Proof of Proposition~\ref{remove3col}]
Suppose we are given a planar digraph $D$ as an instance of the $3$-${\rm lrem}(\vec{P})$-PFC. Let $D'$ be the digraph obtained from $D$ by adding for each vertex $v \in V(D)$ two disjoint copies $\vec{O}_v^+, \vec{O}_v^-$ of the outerplanar acyclic digraph $\vec{O}(\vec{P})$ from Theorem~\ref{outermonchromaticpaths} to $D$ and adding all the arcs $(v,u)$, $u \in V(\vec{O}_v^+)$ and $(u,v)$, $u \in V(\vec{O}_v^-)$. The digraph $D'$ is still planar, since each of the attached copies $\vec{O}_v^+, \vec{O}_v^-, v \in V(D)$ is outerplanar. 

The following claim follows analogously to the proof of Proposition~\ref{remove2col} using the fact that each outerplanar graph has chromatic number at most $3$ and that in each of the copies  $\vec{O}_v^+, \vec{O}_v^-$ for all $v \in V(D)$ all three colors are needed, see Proposition~\ref{outermonchromaticpaths}.
  
	\paragraph{\textbf{Claim}} $D'$ admits a $P$-free $3$-coloring if and only if $D$ admits an ${\rm lrem}(P)$-free $3$-coloring. \\

	Since the sizes of $D$ and $D'$ are linearly related, we have found a polynomial reduction of the $3$-${\rm lrem}(\vec{P})$-PFC to the $3$-$\vec{P}$-PFC. This concludes the proof of the first part of the proposition. For the second claim in the proposition it suffices to verify that $D'$ is acyclic if and only if $D$ is acyclic, since then we can use the same polynomial reduction to also reduce the $3$-${\rm lrem}(\vec{P})$-PFC restricted to acyclic inputs to the $3$-$\vec{P}$-PFC with acyclic inputs. However, this directly follows since $D$ is an induced subdigraph of $D'$, and since each of the copies $\vec{O}_v^+,\vec{O}_v^-$, $v \in V(D)$ of $\vec{O}(\vec{P})$ themselves are acyclic (as guaranteed by Theorem~\ref{outermonchromaticpaths}) and separated from the rest of the graph by directed edge-cuts. This shows the second claim in the Proposition and concludes the proof.
\end{proof}

\section{$P$-free $2$-coloring for paths of length $2$}\label{sec:basecases2}

In this section we show Theorem \ref{thm:2col_NPhard} for $n=3$. 
\begin{proposition}\label{3path2colors}
	For every $\vec{P} \in \mathcal{P}_3$, the $2$-$\vec{P}$-PFC is \NP-hard, even restricted to acyclic inputs.	
\end{proposition}

\begin{proof}

Up to isomorphism there are three different orientations of the path $P_3$. The directed path $\vec{P}_3$, an orientation where the middle vertex of $P_3$ is a sink, denoted by $\vec{V}_3$, and the one where the middle vertex is a source, denoted by $\revvec{V}_3$, see Figure \ref{fig:OrientationP3}. The latter two oriented paths are equivalent up to the reversal of all arcs, and hence it suffices to prove that $\vec{P}_3$-free $2$-coloring and $\vec{V}_3$-free $2$-coloring of planar acyclic digraphs is \NP-hard.
To prove these results, we will reduce the $3$-SAT problem to each of the two problems. For this we need to introduce some gadgets. 

\begin{figure}[htb]
	\centering
	\includegraphics{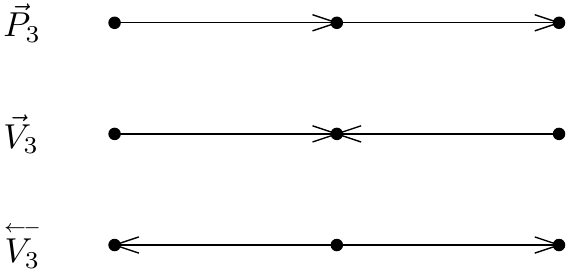}
	\caption{The three different orientations of the path $P_3$.}
	\label{fig:OrientationP3}
\end{figure}

\paragraph{\textbf{\NP-hardness of the $2$-$\vec{P}_3$-PFC}} First we consider the orientation $\vec{P}_3$. The negator gadget as depicted in Figure \ref{fig:negatorP3} forces the vertices $x$ and $y$ to have different colors.
\begin{lemma} \label{lem:negator2P3free}
	The negator gadget in Figure~\ref{fig:negatorP3} has the following properties
	\begin{itemize}
		\item In every $\vec{P}_3$-free 2-coloring of the negator gadget, the vertices $x$ and $y$ must receive different colors. 
		\item There is a $\vec{P}_3$-free 2-coloring of the negator gadget in which the unique incident edges of $x$ and $y$ are bichromatic. 
		\item The negator gadget is acyclic. 
	\end{itemize}
\end{lemma}

\begin{figure}[htb]
	\centering
	\includegraphics[page = 3,scale= 0.7]{figs/Negator} \\[2ex]
	\begin{subfigure}[b]{0.48\textwidth}
		\centering
		\includegraphics[page = 1,scale= 0.7]{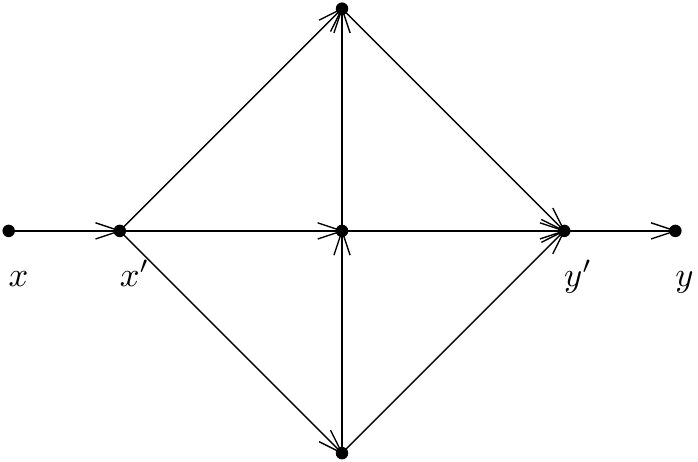} 
		\caption{}
		\label{fig:negatorP3}
	\end{subfigure}
	\begin{subfigure}[b]{0.48\textwidth}
		\centering
		\includegraphics[page = 4,scale = 0.7 ,valign = c]{figs/Negator}
		\caption{}
		\label{fig:NegatorV3}
	\end{subfigure}
\caption{Negator gadget for (a) $\vec{P_3}$-free colorings and (b) for $\vec{V}_3$-free colorings.}
\end{figure}

\begin{proof}
	Assume towards a contradiction that $x$ and $y$ are both colored with the same color $i \in \{0,1\}$. The three vertices $v_1,v_2,v_3$ on the vertical path in the gadget (see Figure~\ref{fig:negatorP3}) induce a $\vec{P}_3$, hence both colors have to be used on these three vertices. The vertex colored with $i$ forces $x'$ and $y'$ to be colored with color $j \neq i$. But there is a vertex colored with $j$ in the vertical path, which is a contradiction since we have a monochromatic induced $\vec{P}_3$.
	Hence $c(x) \neq c(y)$ for all $\vec{P}_3$-free 2-colorings~$c$. 
	It is easy to see that the negator gadget is acyclic and there is a $2$-coloring with bichromatic edges $xx'$ and $yy'$.
\end{proof}

If we have the negator two times in a row, we can force two vertices to have the same color.
We call this gadget as depicted in Figure \ref{fig:extender} the \emph{extender gadget}. We connect the two negator gadgets with endpoint $x_1,y_1$ and $x_2,y_2$ by identifying $y_1$ and $y_2$ such that the horizontal paths have reverse directions. 
\begin{figure}[htb]
	\centering
	\includegraphics[page = 2,scale= 0.7]{figs/Extender} \\[2ex]
	\includegraphics[page = 1,scale= 0.7]{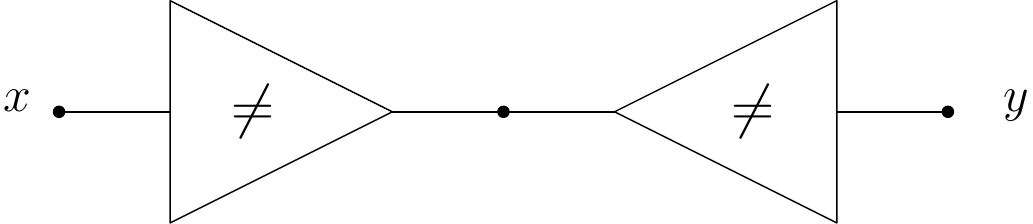}
	\caption{Extender gadget for $\vec{P_3}$-free and $\vec{V}_3$-free colorings.}
	\label{fig:extender}
\end{figure}
It follows directly from Lemma \ref{lem:negator2P3free} that the extender gadget fulfills the following properties.
\begin{corollary}\label{cor:extenderP3} 
	The extender gadget in Figure~\ref{fig:extender} has the following properties
	\begin{itemize}
		\item In every $\vec{P}_3$-free 2-coloring of the extender gadget the vertices $x$ and $y$ must receive the same color.
		\item There is a $\vec{P}_3$-free 2-coloring in which the unique incident edges of $x$ and $y$ are bichromatic. 
		\item The extender gadget is acyclic.
	\end{itemize}
\end{corollary}

Another gadget we will need is the \emph{crossover-gadget}. 
This gadget forces two pairs of antipodal vertices to have the same color. 

\begin{lemma} \label{lem:crossoverP3}
	The crossover gadget as depicted in Figure~\ref{fig:ColoringCrossoverP3} has the following properties:
	\begin{itemize}
		\item In every $\vec{P}_3$-free 2-coloring of the crossover gadget the vertices $x$ and $x'$ as well as the vertices $y$ and $y'$ have the same color. 
		\item For every assignment $p\mathop{:}\{x,x',y,y'\} \to \{0,1\}$ with $p(x) = p(x')$ and $p(y) = p(y')$ there exists a $\vec{P}_3$-free 2-coloring $c$ of the crossover gadget such that $c(x) = c(x')= p(x)= p(x')$, $c(y) = c(y')= p(y)= p(y')$, and all edges incident to $x,x',y,y'$ in the gadget are bichromatic.
		\item The crossover gadget is acyclic. 
	\end{itemize}
\end{lemma}

\begin{figure}[htb]
	\begin{subfigure}[t]{0.48\textwidth}
		\centering
	  	\includegraphics[page = 2,scale= 0.65]{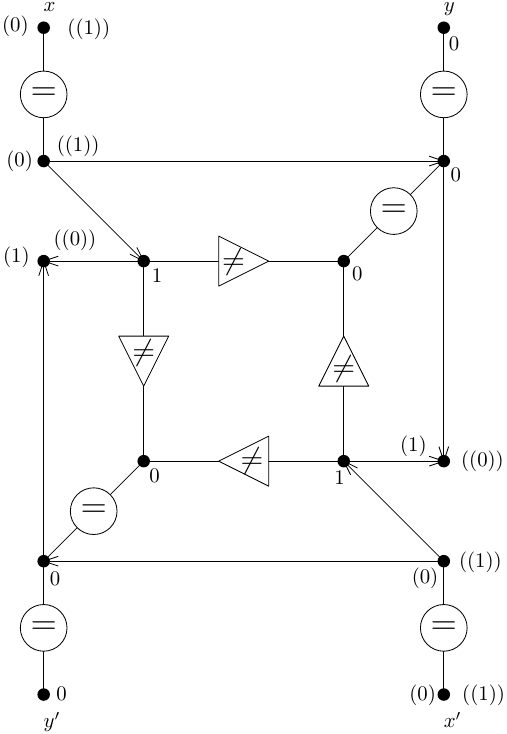}
	  	\caption{}
	  	\label{fig:ColoringCrossoverP3}
	\end{subfigure}
	\begin{subfigure}[t]{0.48\textwidth}
		\centering
	\includegraphics[page = 3,scale = 0.65]{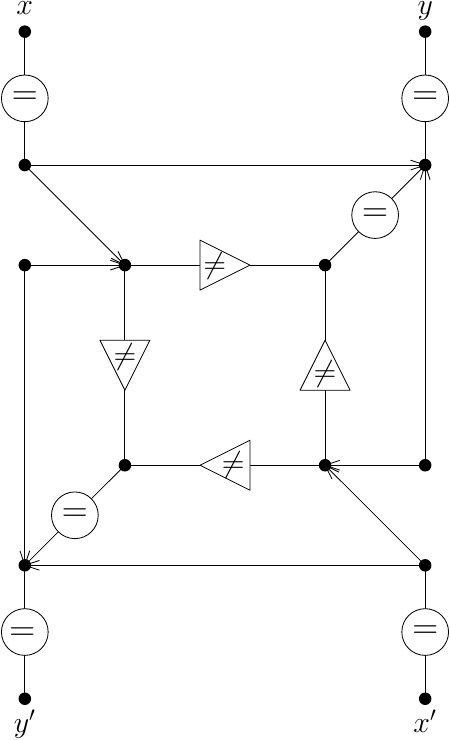}
	\caption{}
	\label{fig:crossoverV3}
\end{subfigure}
\caption{Crossover gadget for (a) $\vec{P}_3$ with possible 2-colorings and (b) crossover gadget of $\vec{V}_3$.}
\end{figure}

\begin{proof}
	Using the negators and extenders of the crossover, we see immediately that $y$ and $y'$ must have the same color, say $0$. 
	If $x$ is colored with $0$, then $x'$ must have color $0$ as well, as otherwise necessarily a monochromatic copy of $\vec{P}_3$ is created.
	Similar $c(x) = 1$ implies $c(x') =1$. For an illustration of the cases see Figure \ref{fig:ColoringCrossoverP3}. 
	Since we can color the extender gadget such that arcs incident to $x,x',y,y'$ are bichromatic and any pre-coloring in which $x$ and $x'$ as well as $y$ and $y'$ have the same color can be extended to a $\vec{P}_3$-free coloring of the gadget, the second property is fulfilled. 
	Moreover, the negator gadgets are arranged in such a way that they cannot contribute to an induced directed cycle. Hence the crossover gadget is acyclic.
\end{proof}

The last gadget we need is the \emph{clause gadget}. 
Assume there is a pre-coloring $c:\{t',x',y',z'\} \rightarrow \{0,1\}$ of some vertices of the gadget.
\begin{lemma} \label{lem:clauseP3}
	The clause gadget as depicted in \ref{fig:clauseGadget} is acyclic and a pre-coloring $c$ can be extended to a $\vec{P}_3$-free $\{0,1\}$-coloring if and only if $c(t') \in \{c(x'),c(y'),c(z')\}$.
\end{lemma}

\begin{figure}[htb]
	\begin{subfigure}{0.48\textwidth}
		\centering
		\includegraphics[page = 1,scale= 0.7]{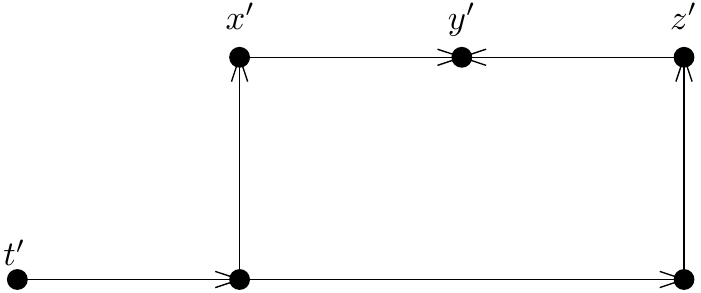}
		\caption{}
		\label{fig:clauseGadget} 
	\end{subfigure}
	\begin{subfigure}{0.48\textwidth}
		\centering
		\includegraphics[page=2,scale = 0.7]{figs/Clause}
		\caption{}
		\label{fig:clauseV3}
	\end{subfigure}
	\caption{Clause gadget for (a) $\vec{P}_3$-free and (b) $\vec{V}_3$-free colorings.}
\end{figure}

\begin{proof}
	Assume that $c(t') = 1$. 
	First we show that there is no $\vec{P}_3$-free 2-coloring of the clause gadget with $c(x') = c(y') = c(z') =0$. 
	Assume there is such a coloring. Because $x'$ and $y'$ (resp. $y'$ and $z'$) together with a third vertex are an induced $\vec{P}_3$, the third vertex has to be colored with $1$. 
	So $t'$ and the other two unlabeled vertices are all colored with 1, contradiction since they build an induced $\vec{P}_3$. 
	It is easy to check that all other combinations lead to a $\vec{P}_3$-free $2$-coloring of the gadget.
\end{proof}

We are now ready to describe the reduction of 3-SAT to $2$-$\vec{P}_3$-PFC.

For a given 3-SAT formula $F = c_1 \wedge c_2 \wedge \ldots \wedge c_k$ with clauses $C = \{ c_1, \ldots , c_k \}$, each consisting of three literals, 
we construct a planar graph $G_F$.
Let $x_1, \ldots, x_m $ denote the literals in $F$ and $\overline{x_1}, \ldots, \overline{x_n}$ their negations.
First we add some vertices for every literal $x_i$ and one for every negation $\overline{x_i}$.
Furthermore we add an additional vertex $t$.
We start connecting the vertices. Every pair $x_i$,  $\overline{x_i}$ is connected by a negator.
So we can be sure that in every $2$-coloring these two vertices have different color. 
For every clause $c_i = x_i \vee y_i \vee z_i$, we add a clause gadget such that $x'$ as depicted in Figure~\ref{fig:clauseGadget} is connected to the vertex corresponding to the literal $x_i$ by an extender gadget. Similarly $y'$ and $y_i$, $z' $ and $z_i$ and $t' $ and $t$.
Note that the graph constructed so far might be non-planar. Still we can draw the graph in such a way that crossings are only between extender gadgets. We can remove those crossings successively by adding a crossover gadget instead. 
Let $x$ and $x'$ two vertices connected by a extender gadget which is crossed by the extender gadget connecting $y$ and $y'$. 
Now we delete the two extender gadgets and add a crossover gadget instead between those four vertices. 
If we do this for every crossing between extender gadgets, we get a planar digraph.

Since all gadgets are acyclic the constructed graph is acyclic as well. Note that we can always choose the orientation of the negator gadget in such a way that they do not form an induced directed cycle. 
We can now conclude the proof that $2$-$\vec{P}_3$-PFC is \NP-hard.

Let us reduce $2$-$\vec{P}_3$-PFC from $3$-SAT. We claim that a given 3-SAT formula $F$ is satisfiable if and only if the graph $G_F$ has a proper $\vec{P}_3$-free $2$-coloring.

If $F$ is satisfiable, we fix an assignment of the literals such that the formula is true. 
We color the vertex $t$ with color $1$. 
Furthermore, we assign the corresponding vertices color $1$ (with color $0$) if the corresponding literals are assigned to be true (false, resp). 
Using the properties of the extender and negator gadgets this coloring is extended to a $2$-coloring of the digraph such that the outer arcs of the negator/extender gadgets are bichromatic. This $2$-coloring is $\vec{P}_3$-free since in the gadgets every $2$-coloring is $\vec{P}_3$-free and the outer arcs are bichromatic.  
	If $F$ is unsatisfiable, for every assignment of $\{0,1\}$ to the literals $x_1, \ldots , x_m$ there is a clause $c_i = x \vee  y \vee  z $ which is not satisfiable, hence $x,y,z$ are assigned with $0$ which corresponds to be colored with $0$. Hence the clause gadget has no $2$-coloring which implies that the constructed digraph has no  $\vec{P}_3$-free $2$-coloring.  

	Since the digraph is constructed in polynomial time in the number of clauses and literals of the 3-SAT formula, this concludes the proof of the \NP-hardness of $2$-$\vec{P}_3$-PFC.\\

Note that a reduction from PLANAR 3-SAT would not simplify the proof since the crossover gadget would still be necessary. To be more precise, the additional vertex $t$, which we introduced in order to fix the color which represents the value truth for the variables and which is connected by an extender gadget to all clause gadgets, might still cause crossings.

\paragraph{\textbf{\NP-hardness of the $2$-$\vec{V}_3$-PFC}}
Analogous to the previous case of the \NP-hardness of $2$-$\vec{V}_3$-PFC, we show the \NP-hardness of $2$-$\vec{V}_3$-PFC by reduction from 3-SAT. 
We define the negator, extender, crossover and clause gadgets similarly.

It is easy to check that the negator gadgets as depicted in \ref{fig:NegatorV3} fulfills the following conditions. \pagebreak[3]

\begin{lemma}\nopagebreak[4] \noindent
	\begin{itemize}
		\item In every $\vec{V}_3$-free 2-coloring of the negator gadget, the vertices $x$ and $y$ must receive different colors. 
		\item In any 2-coloring of a digraph, containing the negator gadget as a subdigraph in such a way that there are only incident edges to $x$ and $y$, there are no monochromatic copies of $\vec{V}_3$ using an edge of the negator gadget and one of the edges incident to $x$ or $y$ which does not belong to the negator gadget. 
		\item The negator gadget is acyclic. 
	\end{itemize}
\end{lemma}

The extender gadget is defined in the same manner as in the case of $\vec{P}_3$ by connecting two negator gadgets, see Figure \ref{fig:extender}. 
In the crossover gadget and the clause gadget, we change the orientation of some arcs as depicted in Figure \ref{fig:crossoverV3} and Figure \ref{fig:clauseV3}.
It is easy to show that this gadgets fulfill the same properties as in the case $\vec{P}_3$, see Corollary \ref{cor:extenderP3}, Lemma \ref{lem:crossoverP3} and Lemma \ref{lem:clauseP3}. 
The reduction to 3-SAT works exactly the same as in the last case. 

Hence we proved the \NP-hardness of $2$-$P$-PFC for every $P \in \mathcal{P}_3$.
\end{proof}

\section{$P$-free $2$-coloring for paths of length $3$}\label{sec:basecases25}
In this section, we prove Theorem~\ref{thm:2col_NPhard} in the case $n=4$. 
\begin{proposition}\label{4path2colors}
	Let $\vec{P} \in \mathcal{P}_4$. Then the $2$-$\vec{P}$-PFC is \NP-hard, even restricted to acyclic inputs.
\end{proposition}

\begin{proof}
There are four non-isomorphic orientations of $P_4$ as illustrated in Figure~\ref{fig:OrientationP4}. Firstly $\vec{P}_4$, the directed path of length $3$, secondly $\vec{N}_4$, the anti-directed path of length $3$, the path $\vec{L}_4$ consisting of vertices $v_1,v_2,v_3,v_4$ and arcs $(v_1,v_2)$, $(v_2,v_3)$, $(v_4,v_3)$, as well as the path $\revvec{L}_4$ obtained from $\vec{L}_4$ by reversing all arcs. Since a given planar acyclic digraph $D$ is $\revvec{L}_4$-free $2$-colorable if and only if the (acyclic) digraph $\revvec{D}$ obtained from $D$ by reversing all arcs is $\vec{L}_4$-free $2$-colorable, it suffices to show the \NP-hardness of the $2$-$\vec{P}$-PFC restricted to acyclic inputs for $\vec{P} \in \{\vec{P_4},\vec{N}_4,\vec{L}_4\}$.

\begin{figure}[htb]
	\centering
	\includegraphics{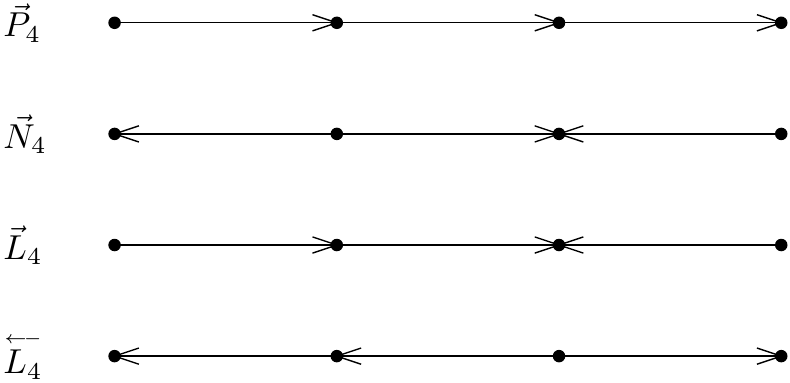}
	\caption{The four different Orientations of the $P_4$.}
	\label{fig:OrientationP4}
\end{figure}

\paragraph{\textbf{\NP-hardness of the $2$-$\vec{N}_4$-PFC}}
We show \NP-hardness by reducing from the $2$-$\vec{V}_3$-PFC restricted to acyclic inputs, which was shown \NP-complete in Section~\ref{sec:basecases2}. Recall that $\vec{V}_3$ is the orientation of the $P_3$ where the middle vertex is a sink. 
Suppose we are given a planar acyclic digraph $D$ as an input to the $2$-$\vec{N}_4$-PFC. Let $D'$ be the planar digraph obtained from $D$ by adding for each vertex $v \in V(D)$ a disjoint copy $\vec{N}_4^v$ of $\vec{N}_4$ and connecting it to $v$ with the arcs $(v,u),$ $u \in V(\vec{N}_4^v)$. We claim that $D'$ is acyclic and that $D$ admits a $\vec{V}_3$-free $2$-coloring if and only if $D'$ has a $\vec{N}_4$-free $2$-coloring.

Firstly, $D'$ is indeed acyclic, since each of the disjoint subgraphs $D$, $\vec{N}_4^v, v \in V(D)$ is acyclic, and no directed cycle can contain vertices from two distinct subgraphs, since no arc in $D'$ emanates from $V(D)$ into any of the copies $\vec{N}_4^v, v \in V(D)$.

To prove the first direction of the equivalence, assume that $D$ admits a $\vec{V}_3$-free $2$-coloring $c\mathop{:}V(D) \rightarrow \{0,1\}$. Then we can extend this to a vertex-$2$-coloring $c'\mathop{:}V(D') \rightarrow \{0,1\}$ of $D'$ by coloring the vertices within each copy $\vec{N}_4^v$ according to a proper $\{0,1\}$-coloring of the (undirected) path $P_4$. We claim that $c'$ is $\vec{N}_4$-free: Suppose towards a contradiction that there is a monochromatic induced copy $x_1, (x_2,x_1), x_2, (x_2, x_3), x_3, (x_4,x_3),x_4$ of $\vec{N}_4$ in $D'$. Since $x_2,x_3,x_4$ induce a monochromatic $\vec{V}_3$ in $D'$, and since $c$ is a $\vec{V}_3$-free $2$-coloring of $D$, there must be $i \in \{2,3,4\}$ such that $x_i \notin V(D)$. Let $v \in V(D)$ be such that $x_i \in V(\vec{P}_4^v)$. 
If $i \in \{2,4\}$, then $(x_{i},x_{i-1}) \in A(D')$, and since there is no arc in $D'$ leaving $V(\vec{N}_4^v)$, we must have $x_{i-1} \in V(\vec{N}_4^v)$ as well. This however contradicts that $c'$ properly colors the copy $\vec{N}_4^v$. 
Hence, $i=3$. Then there is a $j \in \{2,4\}$ such that $x_j \neq v$. As we have $(x_j,x_3) \in A(D')$ and since $v$ is the only in-neighbor of $x_3$ outside $V(\vec{N}_4^v)$, we must have $x_j \in V(\vec{N}_4^v)$ as well. This means that $(x_j,x_2)$ is a monochromatic edge in $\vec{N}_4^v$, again contradicting that $c'$ properly colors this copy. Hence, our assumption was wrong, $c'$ is indeed a $\vec{N}_4$-free $2$-coloring of $D'$.

For the reversed direction, assume that $c':V(D') \rightarrow \{0,1\}$ is a $\vec{N}_4$-free coloring of $D'$, and let $c$ be its restriction to $V(D)$. We claim that $c$ is $\vec{V}_3$-free. Suppose not, then let $x_1,x_2,x_3$ be the vertex-trace of a monochromatic copy of $\vec{V}_3$ in $D$ with sink $x_2$. Since $\vec{N}_4^{x_3}$ is an induced copy of $\vec{N}_4$ in $D'$, it must be bichromatic in the coloring $c'$. Hence, there is $x_4 \in V(\vec{N}_4^{x_3})$ such that $c'(x_4)=c'(x_3)$. This yields that $x_1,x_2,x_3,x_4$ induce a monochromatic copy of $\vec{N}_4$ in the coloring $c'$ of $D'$, a contradiction. This shows that $D$ is indeed $\vec{V}_3$-free $2$-colorable and proves the claimed equivalence.

The correctness of the reduction and the fact that $D'$ can be constructed from $D$ in polynomial time shows that the $2$-$\vec{N}_4$-PFC is \NP-complete for acyclic inputs, as claimed.

\paragraph{\textbf{\NP-hardness of the $2$-$\vec{P}_4$-PFC}}
With the same arguments as in the \NP-hardness proof of the $2$-$\vec{N}_4$-PFC in the last paragraph, it is easy to see that
$2$-$\vec{P}_4$-PFC is \NP-hard. The proof works by reducing from the $2$-$\vec{P}_3$-PFC restricted to acyclic inputs, which was shown \NP-complete in Section~\ref{sec:basecases2}.

\paragraph{\textbf{\NP-hardness of the $2$-$\vec{L}_4$-PFC}}
	We show that deciding whether a given acyclic planar digraph has an $\vec{L}_4$-free $2$-coloring is \NP-hard by reducing from $3$-SAT. For this we need to introduce some gadgets. The first gadget, the \emph{extender gadget} (see Figure~\ref{fig:P4-extender}) enforces the same color on its two end vertices $x$ and $y$.
	\begin{figure}[htb]
	\centering
	\includegraphics[page = 1,scale= 0.7]{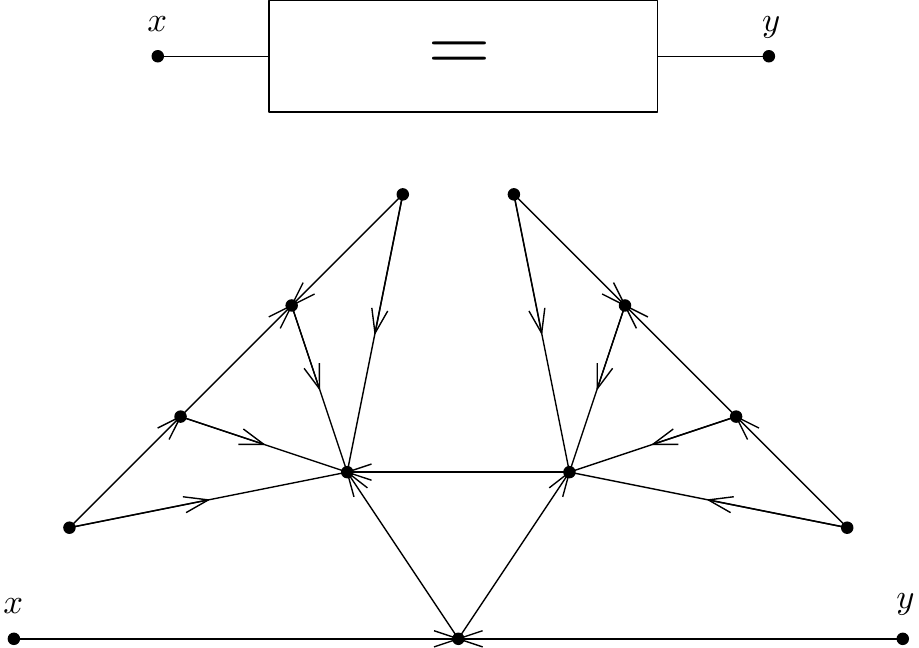}
	\caption{The extender gadget for $\vec{L}_4$-free colorings.}
	\label{fig:P4-extender}
\end{figure}

	The following properties of this gadget are readily verified.

	\begin{lemma}  The extender gadget depicted in Figure~\ref{fig:P4-extender} has the following properties:
		\begin{itemize}
			\item In every $\vec{L}_4$-free $2$-coloring of the extender gadget the vertices $x$ and $y$ must receive the same color.
			\item There is an $\vec{L}_4$-free $2$-coloring of the extender gadget in which the unique incident edges of $x$ and $y$ are bichromatic.
			\item The extender gadget is acyclic.
		\end{itemize}
	\end{lemma}
	\begin{proof}
		Let $c$ be an $\vec{L}_4$-free $2$-coloring of the gadget. Then the color $0$ and $1$ appear in the induced $\vec{L}_4$ in both oriented fans. Since every pair of vertices together with the root vertices $a$ and $b$ of the fan build an induced fan, it holds $c(a) \neq c(b)$.
		The vertex $z$ between $x$ and $y$ is colored with one of the two colors. Without loss of generality assume $c(z) = c(a)$. 
		Hence $x$ and $y$ must receive the color different from $c(z)$ since otherwise they would form a monochromatic induced $\vec{L}_4$. 
		This shows $c(x) = c(y)$ and that the two unique incident edges are bichromatic. 
		Clearly the gadget is acyclic. 
	\end{proof}
	The next gadget is the \emph{negator-gadget} (see Figure~\ref{fig:P4-negator}), which enforces distinct colors on its two end vertices.
	\begin{figure}[htb]
		\centering
		\includegraphics[page = 1,scale= 0.7]{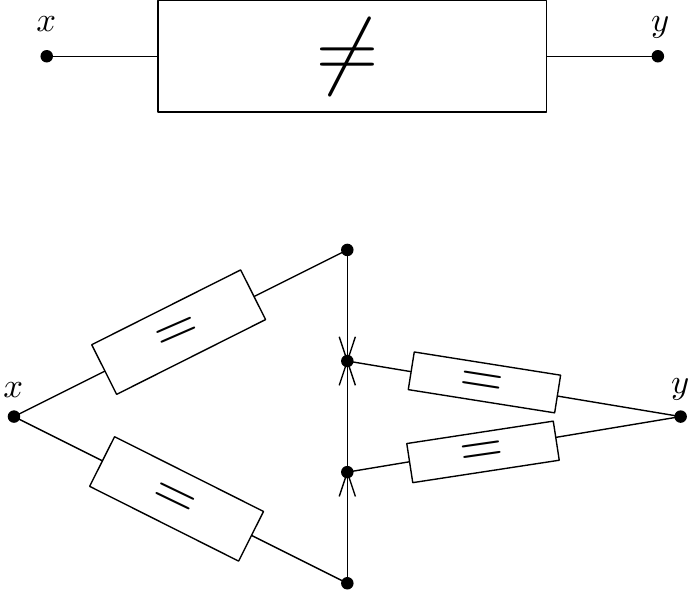}
		\caption{The negator gadget for $\vec{L}_4$-free colorings.}
		\label{fig:P4-negator}
	\end{figure}

	\begin{lemma}
		The negator gadget as depicted in Figure~\ref{fig:P4-negator} has the following properties:
	\begin{itemize}
		\item In every $\vec{L}_4$-free $2$-coloring of the negator gadget, the vertices $x$ and $y$ receive different colors.
		\item There is an $\vec{L}_4$-free $2$-coloring of the negator gadget in which the incident edges of $x$ and $y$ in this gadget are bichromatic.
		\item The negator-gadget is acyclic.
	\end{itemize}
	\end{lemma}
	\begin{proof}
		Assume there is an $\vec{L}_4$-free $2$-coloring $c$ of the negator gadget such that $c(x) = c(y)$. Using the properties of the extender gadget all vertices receive the same color $c(x)$. This gives a monochromatic $\vec{L}_4$ which is a contradiction. 
		The second property follows from the extender gadget where we have bichromatic incident edges. Furthermore the acyclicity of the extender gadget shows that the negator gadget is acyclic. 
	\end{proof}
The last gadget we need is the \emph{crossover gadget} (see Figure~\ref{fig:P4-crossover}), which enforces two pairs of antipodal vertices on its outer face to have the same color.
\begin{figure}[htb]
	\centering
	\includegraphics[page = 1,scale= 0.7]{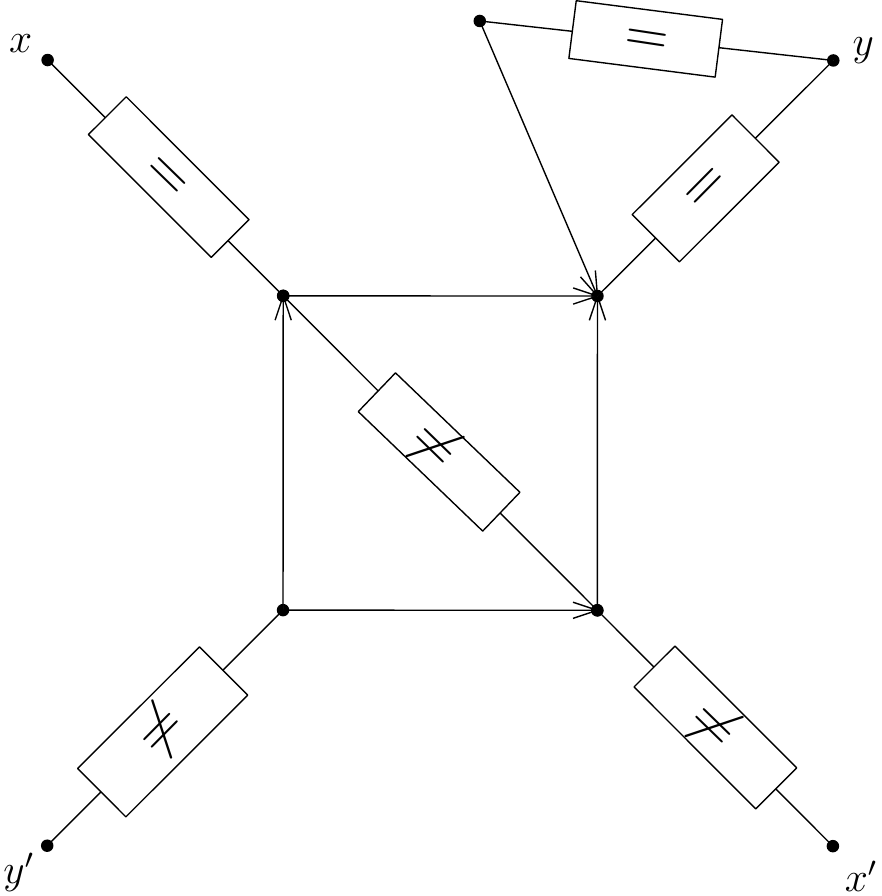}
	\caption{The crossover gadget for $\vec{L}_4$-free colorings.}
	\label{fig:P4-crossover}
\end{figure}

\begin{lemma}\label{obs:crossover}
	The crossover gadget as shown in Figure~\ref{fig:P4-crossover} has the following properties:
\begin{itemize}
\item In every $\vec{L}_4$-free $2$-coloring of the crossover gadget the vertices $x$ and $x'$ as well as the vertices $y$ and $y'$ have the same color.
\item For every $p:\{x,y,x',y'\} \rightarrow \{0,1\}$ such that $p(x)=p(x'), p(y)=p(y')$, there exists a $\vec{L}_4$-free $2$-coloring $c$ of the crossover gadget such that $c(x)=c(x')=p(x)=p(x'), c(y)=c(y')=p(y)=p(y')$, and such that all incident edges of $x, y, x', y'$ in the gadget are bichromatic.
\item The crossover gadget is acyclic.
\end{itemize}
\end{lemma}
\begin{proof}
	The extender gadgets immediately show that $x$ and $x'$ have the same color in every $\vec{L}_4$ $2$-coloring. Furthermore in order to avoid a monochromatic $\vec{L}_4$ the vertices $y$ and $y'$ receive the same color. 
	The second and third property follows from the properties of the negator and extender gadget. 
\end{proof}
We are now ready to describe the reduction of $3$-SAT to the $2$-$\vec{L}_4$-PFC with acyclic inputs. Suppose we are given a $3$-SAT formula on the literals $x_i, \overline{x_i},$ $i=1,\ldots,n$ and consisting of $m$ clauses $c_1,\ldots,c_m$ as an input to $3$-SAT. We construct an auxiliary acyclic digraph $D'$, which might not yet be planar as follows: We have $2n$ vertices corresponding to $x_i, \overline{x_i}, i=1,\ldots,n$, and connect $x_i$ and $\overline{x_i}$ by a negator. For every $j=1,\ldots,m$, we add a disjoint copy of $\vec{L}_4$ on the vertices $x_{1,j}, x_{2,j},x_{3,j},x_{4,j}$ and arcs $(x_{1,j},x_{2,j}),(x_{2,j},x_{3,j}), (x_{4,j},x_{3,j})$. Letting $c_j=l_1^j \vee l_2^j \vee l_3^j$, for $i=1,2,3$ we connect $x_{i,j}$ with an extender to the vertex representing the literal $l_i^j$. We add a further special vertex $t$ which is connected via extenders to each of the vertices $x_{4,j}$, $j=1,\ldots,m$. Clearly, the so-defined digraph $D'$ can be constructed in polynomial time in $m$ and $n$ and is of polynomial size in $m$ and $n$. It is furthermore clear from the properties of the extender and negator gadgets that $D'$ is an acyclic digraph.

We claim that $D'$ admits an $\vec{L}_4$-free $2$-coloring if and only if $c_1 \wedge \ldots \wedge c_m$ is satisfiable. For the first direction, suppose we are given an $\vec{L}_4$-free $\{0,1\}$-coloring $c$ of $D'$. W.l.o.g. let $c(t)=0$. Then by the properties of the negators, we have $c(x_i)=1-c(\overline{x_i}) \in \{0,1\}$ for $i=1,\ldots,m$. We claim that assigning the truth value $c(x_i)$ to each variable $x_i$ defines a truthful assignment for $c_1 \wedge \ldots \wedge c_m$. Suppose not, then there exists $j \in \{1,\ldots,m\}$ such that the colors of the vertices representing the three literals $l_1^j, l_2^j, l_3^j$ in $c_j$ are $0$ each. Since these vertices are connected by extenders to the vertices $x_{1,j},x_{2,j},x_{3,j}$ of $D'$, we have $c(x_ {i,j})=0$, $i=1,2,3$. We further have $c(x_{4,j})=c(t)=0$, so that $x_{1,j},x_{2,j},x_{3,j},x_{4,j}$ induce a monochromatic copy of $\vec{L}_4$ in $D'$, a contradiction. Hence, we have indeed found a truthful assignment.

For the reversed direction, suppose we are given an assignment of $0,1$-truth-values to the variables $x_1,\ldots,x_n$. We define a $\{0,1\}$-coloring of $D'$ by assigning to each vertex representing a literal its truth value, and coloring $t$ with color $0$. Further, every vertex $x_{i,j}$, $i=1,2,3, j=1,\ldots,m$ is colored with the truth value of the literal it is connected to by an extender, and the vertices $x_{4,j}$, $j=1,\ldots,m$ receive color $0$. The partial coloring defined so far has the property that the two end vertices of any extender in $D'$ have the same color, while end vertices of negator gadgets have distinct colors. Using the second property of extender and negator gadgets, we can extend this partial coloring to a full $\{0,1\}$-coloring $c$ of $D'$ by coloring the internal vertices of every extender or negator gadget by an $\vec{L}_4$-free coloring such that the incident edges of the two ends of such a gadget are bichromatic. We claim that the so-defined coloring $c$ is a $\vec{L}_4$-free coloring of $D$. Suppose not, then there must be an induced copy $P$ of $\vec{L}_4$ in $D'$ which is monochromatic under $c'$. We first observe that $P$ cannot contain any internal vertices of extender or negator gadgets. Indeed, since the coloring of the internal vertices of any gadget by definition is $\vec{L}_4$-free, if $P$ intersects the interior of a gadget, it would have to contain one of the two end vertices of the gadget and hence an incident edge of these vertices as well. This contradicts the fact that all such edges are by definition bichromatic. Hence, $P$ is contained in the digraph obtained from $D$ by deleting all internal vertices of extender and negator gadgets. The only non-singleton components of this digraph are constituted by the $m$ disjoint copies of $\vec{L}_4$ induced by $x_{1,j},x_{2,j},x_{3,j},x_{4,j}$, $j=1,\ldots,m$. This means that $P$ must equal one of these paths, i.e., $c(x_{1,j})=c(x_{2,j})=c(x_{3,j})=c(x_{4,j})=0$ for some $j$. By definition, this means that all three literals in the clause $c_j$ have truth value $0$, a contradiction. Hence, we proved the claimed equivalence. 

The digraph $D'$ might not be planar. In order to solve this issue, we consider a drawing of $D'$ in the plane in which the vertices $t,x_1,\overline{x_1},\ldots,x_n,\overline{x_n}$ are placed on a horizontal line $\ell_1$ such that the negator-gadgets between the vertices $x_i, \overline{x_i}$ are drawing as thin horizontal boxes on top of the line $\ell_1$ and the vertices $x_{i,j}$, $i=1,2,3, j=1,\ldots,m$ are placed on a parallel horizontal line $\ell_2$ such that the $m$ disjoint paths $x_{1,j},x_{2,j},x_{3,j},x_{4,j}$, $j=1,\ldots,m$ are drawn crossing-free on $\ell_2$. All the extender gadgets now connect vertices on the line $\ell_1$ to the line $\ell_2$, so we can draw them within thin rectangular strips touching $\ell_1$ and $\ell_2$. The only possible crossings in this drawing are between pairs of such extender-strips spanned between $\ell_1$ and $\ell_2$. Note that the number of pairs of such strips which cross is $O(|V(D')|^4)=\poly(m,n)$.

We now sequentially transform this drawing into a drawing of a planar digraph $D$: As long as there is a crossing between two extenders-strips connecting vertices $a_1, b_1$ and $a_2,b_2$ in the drawing we locally replace the crossing by a \emph{crossover} gadget and connect, with four disjoint extender gadgets, $a_1$ to the outer vertex $x$ of the crossover gadget, $b_1$ to $x'$, $a_2$ to $y$, $b_2$ to $y'$. This can be done such that the crossover gadget and the four new extender gadgets do not intersect pairwise and do not intersect with any other features of the drawing. Performing this operation sequentially for all crossings between extender-strips in the drawing, we construct a planar digraph $D$ of polynomial size in $m$ and $n$. Since $D'$ and each of the extender and crossover gadgets are acyclic, $D$ is acyclic. Further, whenever we replace a pair of intersecting extender gadgets by four non-intersecting extender gadgets and a central crossover gadget, this has the same effect on transporting colors as the two original extender gadgets it replaces, hence every $\vec{L}_4$-free coloring of $D$ gives rise to a truthful assignment of the formula $c_1\wedge c_2\wedge \ldots \wedge c_m$ by the same arguments as above for $D'$. Vice versa, if $c_1\wedge c_2\wedge \ldots \wedge c_m$ has a truthful assignment, then there is a $\vec{L}_4$-free coloring $c'$ of $D'$. By the last item of Observation~\ref{obs:crossover}, for every crossover gadget and for every extender gadget we used to replace a pair of crossing extender gadgets of $D'$, we can color these gadgets such that there are no monochromatic copies of $\vec{L}_4$ in the interior of any of the gadgets and such that all edges connecting an internal vertex of a gadget to the outside is bichromatic. Adding these colorings to the coloring of the vertices in $D'$ described by $c'$, we obtain a $2$-coloring $c$ of $D$. No monochromatic copy of $\vec{L}_4$ in $D$ with respect to the coloring $c$ can contain an internal vertex of one of these gadgets. Hence every such copy would have existed already in the coloring $c'$ of the digraph $D'$. This is a contradiction, which proves that $D$ is $\vec{L}_4$-free $2$-colorable. Summarizing, we have shown that $D$ is $\vec{L}_4$-free $2$-colorable iff $c_1\wedge c_2\wedge \ldots \wedge c_m$ admits a truthful assignment. Since we can construct the polynomially sized planar and acyclic digraph $D$ from the formula in polynomial time in $m$ and $n$, this concludes the desired reduction showing that $3$-SAT polynomially reduces to the $2$-$\vec{L}_4$-PFC with acyclic inputs.
This shows that the $2$-$\vec{L}_4$-PFC is \NP-hard, concluding the proof of Proposition~\ref{4path2colors}.
\end{proof}

We are now ready for the proof of Theorem~\ref{thm:2col_NPhard}.
\begin{proof}[Proof of Theorem~\ref{thm:2col_NPhard}]
	If $n=1$ and $\vec{P}$ is the one-vertex-path, then the $2$-$\vec{P}$-PFC admits a trivial constant-time algorithm. If $n=2$ and $\vec{P}$ is the directed edge, then the $2$-$\vec{P}$-PFC amounts to testing whether the underlying graph of a given digraph $D$ is bipartite, which can be checked in polynomial time.

	For every $n \ge 3$ and every $\vec{P} \in \mathcal{P}_n$, the $2$-$\vec{P}$-PFC clearly is contained in \NP, since we can verify the correctness of a $\vec{P}$-free coloring of a given digraph $D$ in time $O(|V(D)|^{|V(\vec{P})|})$ by brute-force.

	We now prove the \NP-hardness of the $2$-$\vec{P}$-PFC, restricted to acyclic inputs, for all $\vec{P} \in \mathcal{P}_n$ and $n \ge 3$ by induction on $n$. The base cases $n \in \{3,4\}$ are covered by Propositions ~\ref{3path2colors},~\ref{4path2colors}. So assume for the inductive step that $n \ge 5$, $\vec{P} \in \mathcal{P}_n$ and that we have shown that $2$-$\vec{P}$-PFC with acyclic inputs is \NP-hard for all oriented paths $\vec{P}$ of length at least two and at most $n-1$. Then also the $2$-${\rm lrem}(\vec{P})$-PFC restricted to acyclic inputs is \NP-hard, since ${\rm lrem}(\vec{P})$ is a path on $n-2 \ge 3$ vertices. Proposition~\ref{remove2col} now implies the \NP-hardness of the $2$-$\vec{P}$-PFC, restricted to acyclic inputs. This concludes the proof by induction.
\end{proof}

\section{$P$-free $3$-coloring for paths of length $1$ and $2$}\label{sec:basecases3} 

In this section, we prove Theorem~\ref{thm:3col_NPhard} for $n \in \{2,3\}$. Together with Proposition~\ref{remove3col} we will then be able to prove Theorem~\ref{thm:3col_NPhard} in its full generality.
We need the following easy consequence of Theorem~\ref{outermonchromaticpaths}. By $\vec{V}_3$ we denote the oriented path of length two whose middle vertex is a source.

\begin{figure}[htb]
	\centering
	\includegraphics[scale= 0.6, page =3]{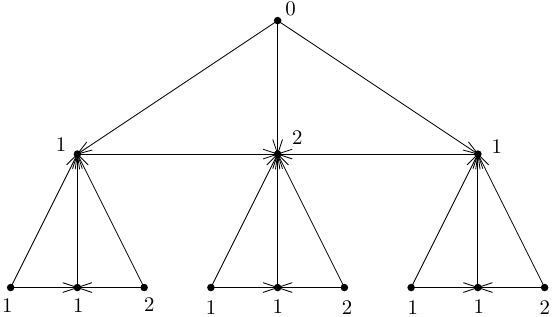}
	\caption{The acyclic outer planar digraph $\vec{O} = \vec{O}(\vec{V}_3) $ with a $3$-coloring such that there is unique vertex of color $0$.}
	\label{fig:OuterplanarV3crit}
\end{figure}
\begin{observation}\label{criticaloutermonochromaticpaths}
	The acyclic outerplanar digraph $\vec{O} = \vec{O}(\vec{V}_3) $ as introduced in Section~\ref{sec:outer} and illustrated in Figure~\ref{fig:OuterplanarV3crit} has the following properties 
	\begin{itemize}
		\item has no $\vec{V}_3$-free $2$-coloring, but
		\item there is a $\vec{V}_3$-free coloring with colors $\{0,1,2\}$ with a unique vertex $x_0 \in V(\vec{O})$ of color $0$.
	\end{itemize}
\end{observation}

\begin{proposition}\label{23path3colors}
	Let $\vec{P} \in \mathcal{P}_2 \cup \mathcal{P}_3$. Then the $3$-$\vec{P}$-PFC is \NP-hard, even restricted to acyclic inputs.
\end{proposition}

\begin{proof}
	If $P \in \mathcal{P}_2$, then $P$ is simply a directed edge between two vertices. Hence, a $P$-free $3$-coloring of a planar digraph $D$ is equivalent to a proper $3$-coloring of its underlying planar graph. The claim now directly follows from the well-known fact that deciding whether a given planar graph is properly $3$-colorable is \NP-hard~\cite{garey}, and since we can provide a given graph with an acyclic orientation in polynomial time.

	Now let $P \in \mathcal{P}_3$. Then $P$ is isomorphic to either the directed path $\vec{P}_3$, $\vec{V}_3$ or $\revvec{V}_3$ of $P_3$, see Figure \ref{fig:OrientationP3}. 
	Again we may assume $P \in \{\vec{P}_3,\vec{V}_3\}$.

	\paragraph{\textbf{\NP-hardness of the $3$-$\vec{V}_3$-PFC}}
	We again show \NP-hardness by reducing from the $3$-colorability of undirected planar graphs. 
	Suppose we are given a planar graph $G$ as an input to the $3$-colorability problem. Let $D$ be an acyclic orientation of $G$, and let $D'$ be the planar digraph obtained from $D$ by adding for each vertex $v \in V(D)$ a disjoint copy $\vec{O}_v$ of the acyclic outerplanar digraph $\vec{O}$ as given by Observation~\ref{criticaloutermonochromaticpaths} and connecting it to $v$ with the arcs $(u,v),$ $u \in V(\vec{O}_v)$. We claim that $\chi(G) \le 3$ if and only if $D'$ has a $\vec{V}_3$-free $3$-coloring. 
	
	To prove the first direction, assume that $G$ admits a proper $3$-coloring $c\mathop{:}V(G)=V(D) \rightarrow \{0,1,2\}$. Then we can extend this to a vertex-$3$-coloring $c'\mathop{:}V(D') \rightarrow \{0,1,2\}$ of $D'$ by coloring the vertices within each copy $\vec{O}_v$ according to a $\vec{V}_3$-free $\{0,1,2\}$-coloring of $\vec{O}$ in which only one vertex receives color $c(v)$ (the existence of such a coloring is guaranteed by Observation~\ref{criticaloutermonochromaticpaths}). 
	
	We claim that $c'$ is $\vec{V}_3$-free: Suppose towards a contradiction there was a monochromatic copy of $\vec{V}_3$ in $D'$ induced by the vertices $x_1,x_2,x_3$ of color $i \in \{0,1,2\}$. Let $x_2$ be the sink of this copy. If $x_2 \in V(D')\setminus V(D)$, then we have $\{x_1,x_2,x_3\} \subseteq \{x_2 \} \cup N_{D'}^+(x_2) \subseteq O(\vec{O}_v)$ for some $v \in V(D)$, contradicting that by definition $c'|_{\vec{O}_v}$ is $\vec{V}_3$-free. So we must have $x_2 \in V(D)$. Since $c$ is a proper coloring of $G$ and $c(x_1)=c(x_2)=c(x_3)=i$, we must have $x_1, x_3 \in O(\vec{O}_{x_2})$. However, this contradicts our definition of $c'$, according to which there is exactly one vertex in $\vec{O}_{x_2}$ of color $c(x_2)=i$ under $c'$. This contradiction shows that our assumption was wrong, indeed, $c'$ is a $\vec{V}_3$-free coloring of $D'$.
	
	For the reverse, assume that $c'\mathop{:}V(D') \rightarrow \{0,1,2\}$ is a $\vec{V}_3$-free coloring of $D'$, and let $c$ be its restriction to $V(D)=V(G)$. If this was no proper coloring of $G$, there would be an arc $(v_1,v_2) \in A(D)$ with $c'(v_1)=c(v_1)=c(v_2)=c'(v_2)\mathop{=:}i$. Let $u_2$ be a vertex in $V(\vec{O}_{v_2})$ such that $c'(u_2)=c(v_2)=i$ (such a vertex must exist by Observation~\ref{criticaloutermonochromaticpaths}). Then the vertices $v_1,v_2,u_2$ induce a monochromatic copy of $\vec{V}_3$ in $D'$, contradicting our choice of $c'$. This shows that indeed $c$ is a proper $3$-coloring of $G$ and hence $\chi(G) \le 3$.
	
	Since the digraph $D'$ can be constructed from $G$ in polynomial time, the above equivalence yields a reduction of the $3$-coloring problem on planar graphs to the $3$-$\vec{V}_3$-PFC on acyclic planar digraphs. This concludes the proof.

\paragraph{\textbf{\NP-hardness of the $3$-$\vec{P}_3$-PFC}}	
	We show \NP-hardness by reducing from the $3$-colorability of undirected planar graphs. The proof is analogous to the \NP-hardness of $3$-$\vec{V}_3$-PFC which was proven in the last paragraph. 

\end{proof}

We are now ready for the proof of Theorem~\ref{thm:3col_NPhard}.
\begin{proof}[Proof of Theorem~\ref{thm:3col_NPhard}]
	If $n=1$ and $P$ is the one-vertex-path, then the $3$-$P$-PFC admits a trivial constant-time algorithm.

	For every $n \ge 2$ and every $P \in \mathcal{P}_n$, the $3$-$P$-PFC clearly is contained in \NP, since we can verify the correctness of a $P$-free coloring of a given digraph $D$ in time $O(|V(D)|^{|V(P)|})$ by brute-force.

	We now prove the \NP-hardness of the $3$-$P$-PFC, restricted to acyclic inputs, for all $P \in \mathcal{P}_n$ and $n \ge 2$ by induction on $n$. The base cases $n \in \{2,3\}$ are covered by Proposition~\ref{23path3colors}. So assume for the inductive step that $n \ge 4$, $P \in \mathcal{P}_n$ and that we have shown that $3$-$P$-PFC with acyclic inputs is \NP-hard for all oriented paths $P$ of length at least two and at most $n-1$. Then also the $3$-${\rm lrem}(P)$-PFC restricted to acyclic inputs is \NP-hard, since ${\rm lrem}(P)$ is a path on $n-2 \ge 2$ vertices. Proposition~\ref{remove3col} now implies the \NP-hardness of the $3$-$P$-PFC, restricted to acyclic inputs. This concludes the proof by induction.
\end{proof}

\bibliographystyle{abbrv}
\bibliography{bib}
\end{document}